\documentclass[12pt]{amsart}
\usepackage{amsfonts,amssymb,amscd,amsmath,enumerate,verbatim,color}
\usepackage[latin1]{inputenc}
\usepackage{amscd}
\usepackage{latexsym}

\usepackage{graphicx} 

\usepackage{mathptmx}

\usepackage{hyperref}
\def\NZQ{\mathbb}               

\def\ZZ{{\NZQ Z}}
\def\RR{{\NZQ R}}

\def\frk{\mathfrak}               

\def\Phi{{\frk N}}
\def\ab{{\mathbf a}}

\def\eb{{\mathbf e}}

\def\xb{{\mathbf x}}
\def\yb{{\mathbf y}}

\def\opn#1#2{\def#1{\operatorname{#2}}} 
\opn\gr{gr}

\def\Ac{{\mathcal A}}
\def\Bc{{\mathcal B}}

\def\Hc{{\mathcal H}}
\def\Sc{{\mathcal S}}

\def\Pc{{\mathcal P}}

\def\Rc{{\mathcal R}}

\def\Vol{{\textnormal{Vol}}}

\newtheorem{theorem}{Theorem}[section]
\newtheorem{lemma}[theorem]{Lemma}
\newtheorem{corollary}[theorem]{Corollary}
\newtheorem{proposition}[theorem]{Proposition}

\theoremstyle{definition}
\newtheorem{remark}[theorem]{Remark}

\newtheorem{example}[theorem]{Example}

\newtheorem{problem}[theorem]{Problem}
\newtheorem{conjecture}[theorem]{Conjecture}

\let\epsilon\varepsilon
\let\phi=\varphi
\let\kappa=\varkappa
\opn\dis{dis}
\opn\height{height}
\opn\dist{dist}
\def\pnt{{\raise0.5mm\hbox{\large\bf.}}}

\opn\Lex{Lex}
\opn\conv{conv}

\begin{document}

\title{Symmetric edge polytopes and matching generating polynomials}
\author{Hidefumi Ohsugi and Akiyoshi Tsuchiya}
\address{Hidefumi Ohsugi,
	Department of Mathematical Sciences,
	School of Science,
	Kwansei Gakuin University,
	Sanda, Hyogo 669-1337, Japan} 
\email{ohsugi@kwansei.ac.jp}

\address{Akiyoshi Tsuchiya,
Graduate school of Mathematical Sciences,
University of Tokyo,
Komaba, Meguro-ku, Tokyo 153-8914, Japan} 
\email{akiyoshi@ms.u-tokyo.ac.jp}

%


\begin{abstract}
 Symmetric edge polytopes $\mathcal{A}_G$ of type A are lattice polytopes arising from the root system $A_n$ and finite simple graphs $G$. 
There is a connection between $\mathcal{A}_G$ and the Kuramoto synchronization model in physics. 
In particular, the normalized volume of $\mathcal{A}_G$ plays a central role. 
In the present paper, we focus on a particular class of graphs.
In fact, for any cactus graph $G$, we give a formula for the $h^*$-polynomial of $\mathcal{A}_{\widehat{G}}$ by using matching generating polynomials, where $\widehat{G}$ is the suspension of $G$.  
This gives also a formula for the normalized volume of $\mathcal{A}_{\widehat{G}}$.
Moreover, via methods from chemical graph theory, we show that for any cactus graph $G$, the $h^*$-polynomial of $\mathcal{A}_{\widehat{G}}$ is real-rooted.
 Finally, we extend the discussion to symmetric edge polytopes of type $B$, which are lattice polytopes arising from the root system $B_n$ and finite simple graphs.\\
\end{abstract}

\maketitle

\section{Introduction}
A \textit{lattice polytope} $\Pc \subset \RR^n$ is a convex polytope all of whose vertices have integer coordinates. 
Many lattice polytopes arising from graphs are constructed and have been studied from several viewpoints. 

In \cite{CSC}, lattice polytopes arising from the root system of type $A_n$ and finite graphs were introduced. Let $G$ be a finite simple undirected graph on the vertex set $[n]:=\{1,\ldots,n\}$ with the edge set $E(G)$. The \textit{symmetric edge polytope $\Ac_G$} (\textit{of type $A$)} of $G$ is the lattice polytope which is the convex hull of
\[
 \{ \pm (\eb_i - \eb_j) : \{i,j\} \in E(G) \},
\]
where $\eb_i$ is the $i$-th unit coordinate vector in $\RR^n$. 
Symmetric edge polytopes, in particular, their Ehrhart polynomials were studied from a viewpoint of algebraic number theory. 
In fact, the Ehrhart polynomials of special symmetric edge polytopes have properties similar to Riemann's $\zeta$ function \cite{localRiemann,zeroroot}. 
Moreover, many results about zero loci of the Ehrhart polynomials of symmetric edge polytopes have been found from a viewpoint of algebraic combinatorics 
\cite{interlacingEhrhart,fivemen,OS}. 
On the other hand, symmetric edge polytopes are always reflexive polytopes, i.e., their dual polytopes are also lattice polytopes. 
Reflexive polytopes correspond to Gorenstein toric Fano varieties and they   give many explicit constructions of mirrors of Calabi--Yau hypersurfaces \cite{mirror}.
By a work of Hibi \cite{hibi}, it follows that their $h^*$-polynomials of reflexive polytopes are always palindromic.
One of current hot topics on the study of lattice polytopes is unimodality questions for $h^*$-polynomials.
Since any symmetric edge polytope has a regular unimodular triangulation, its $h^*$-polynomial is unimodal due to a result of Bruns and R\"{o}mer \cite{BR}.
Moreover, the $h^*$-polynomial of $\Ac_G$ coincides with the $h$-polynomial of a unimodular triangulation of the boundary $\partial \Ac_G$.
From a viewpoint of algebraic and topological combinatorics, Gal \cite{Gal} conjectured that the $h$-polynomial of a flag triangulation of a sphere is $\gamma$-positive, which directly implies the unimodality.
More strongly, Nevo and Petersen conjectured that the $\gamma$-polynomial of the $h$-polynomial of a flag triangulation of a sphere coincides with the $f$-polynomial of a balanced simplicial complex \cite{NP}.
In \cite{HJMsymmetric,OTlocally}, it is shown that the $h^*$-polynomials of the symmetric edge polytopes of certain classes of graphs are $\gamma$-positive.
On the other hand, the $h^*$-polynomial of $\Ac_G$ is not always real-rooted (Example \ref{ex:nonreal}). 
Note that if a polynomial all of whose coefficients are positive is palindromic and real-rooted, then it is $\gamma$-positive and log-concave.

Recently, the normalized volumes of symmetric edge polytopes have attracted much attention.
The author of \cite{Chen} calls symmetric edge polytopes \textit{adjacency polytopes} and refer to the normalized volumes as the \textit{adjacency polytope bounds}.
Adjacency polytopes appeared in the context of the Kuramoto model, describing the behavior of interacting oscillators \cite{Kuramoto}. 
There are many applications of the Kuramoto model in several fields of study in biology, physics, chemistry, engineering, and social science.
In many cases, an adjacency polytope bound gives an upper bound of the number of possible solutions in the Kuramoto equations \cite{Chen}. In \cite{CDM, DDM, OS}, explicit formulas of the adjacency polytope bounds, i.e., the normalized volumes of the symmetric edge polytopes of certain classes of graphs are given. 
Moreover, in \cite{DC}, another type of adjacency polytope bounds is discussed.

In the present paper, from the above background we investigate the normalized volume and the $\gamma$-polynomial of the $h^*$-polynomial of a symmetric edge polytope.
In particular, we focus on the suspension $\widehat{G}$ of a graph $G$.
Here $\widehat{G}$ is the graph on the vertex set $[n+1]$ and the edge set
$E(G) \cup \{ \{i,n+1 \} : i \in [n]\}$.
In \cite{OTlocally}, a formula of the $\gamma$-polynomial of the $h^*$-polynomial of $\Ac_{\widehat{G}}$ by using interior polynomials, which are a version of the Tutte polynomials for hypergraphs introduced by K\'alm\'an \cite{interior}, is given. Moreover, the $h^*$-polynomial of $\Ac_{\widehat{G}}$ is always $\gamma$-positive. Furthermore, this formula also gives a formula of the normalized volume of $\Ac_{\widehat{G}}$.

Our first main theorem is another formula of the $\gamma$-polynomial of the $h^*$-polynomial of $\Ac_{\widehat{G}}$ of a certain class of graphs $G$ by using matching generating polynomials.
A \textit{cactus graph} is a graph $G$ such that each edge of $G$ belongs to at most one cycle of $G$.

\begin{theorem}
\label{formulaA}
Let $G$ be a cactus graph on $[n]$ and let $g(G,x)$ be the matching generating polynomial of $G$.
Then the $\gamma$-polynomial of the $h^*$-polynomial of $\Ac_{\widehat{G}}$ is 
$$
g(G,2x) + \sum_{R \in {{\mathcal R}_2'}(G)} (-2)^{c(R)} g (G-R , 2x) \ 
x^{\frac{|E(R)|}{2} },
$$
where ${\mathcal R}_2' (G)$ is the set of all subgraphs of $G$ 
consisting of vertex-disjoint even cycles,
and $c(R)$ is the number of the cycles of $R$.
	Moreover, the normalized volume of ${\mathcal A}_{\widehat{G}}$ is 
$$
2^n g (G,  1/2) + \sum_{R \in {{\mathcal R}_2'}(G)} (-2)^{c(R)} 2^{n-|E(R)|} g (G-R ,1/2).
$$
\end{theorem}

In \cite{DDM}, the normalized volume of the symmetric edge polytope of a wheel graph is computed.
Note that a wheel graph is the suspension of a cactus graph.
By using Theorem~\ref{formulaA}, we can give explicit formulas of the normalized volume and the $\gamma$-polynomial of the $h^*$-polynomial of the symmetric edge polytope of a wheel graph (Example~\ref{wheelexample}).  
On the other hand, applying Theorem~\ref{formulaA} to a graph $G$ which has no even cycle, we can describe the $\gamma$-polynomial of the $h^*$-polynomial of $\Ac_{\widehat{G}}$ by a single matching generating polynomial and we know that it is real-rooted (Corollary~ \ref{cor:noeven}). 
Our second main theorem extends the real-rootedness of this result to the suspension of any cactus graph via methods from chemical graph theory.
In fact,

\begin{theorem}
\label{thm:realrooted}
Let $G$ be a cactus graph.
Then the $h^*$-polynomial of $\Ac_{\widehat{G}}$ is real-rooted.
\end{theorem}

In the present paper, we also discuss Nevo--Petersen's conjecture for $\Ac_{\widehat{G}}$. More precisely, we show the following.

\begin{theorem}
\label{thm:flag_A}
Let $G$ be a graph which has no even cycles.
Then the $\gamma$-polynomial of $h^*(\Ac_{\widehat{G}},x)$ coincides with the $f$-polynomial of a flag simplicial complex.
\end{theorem}
Note that the $h$-polynomial of a flag simplicial complex coincides with that of a balanced simplicial complex \cite{flagcomplex}.
Hence Theorem \ref{thm:flag_A} proves that Nevo--Petersen's conjecture holds for any flag unimodular triangulation of the boundary $\partial\Ac_{\widehat{G}}$ in this case.
We also show that for a forest $G$, $\partial\Ac_{\widehat{G}}$ has a flag unimodular triangulation by using the algebraic technique of Gr\"{o}bner bases from
Theorem \ref{flagtri} (since a forest is a bipartite graph with no cycles).

In \cite{OTinterior}, lattice polytopes arising from the root system of type $B_n$ and finite graphs were introduced. The \textit{symmetric edge polytope $\Bc_G$ of type $B$} of  a graph $G$ on the vertex set $[n]$ is the lattice polytope which is the convex hull of
\[
\{\pm \eb_i: 1 \leq i \leq n\} \cup \{ \pm \eb_i \pm \eb_j : \{i, j \} \in E(G)\}.
\]
Then it follows that $\Bc_G$ is reflexive if and only if $G$ is a bipartite graph. In the case, $\Bc_G$ has a regular unimodular triangulation. Moreover, if $G$ is bipartite, the $\gamma$-polynomial of the $h^*$-polynomial of $\Bc_G$ can be described by an interior polynomial.
Similarly to the case of symmetric edge polytopes of type A,
we give a formula of the $\gamma$-polynomial of the $h^*$-polynomial of $\Bc_G$ for a cactus bipartite graph $G$ by using matching generating polynomials and prove that the $h^*$-polynomial is real-rooted (Theorem \ref{thm:cactusB}).
Moreover, we show that for a forest $G$, the $\gamma$-polynomial of the $h^*$-polynomial of $\Bc_G$ coincides with the $f$-polynomial of a flag simplicial complex (Theorem \ref{thm:NPB}).
Namely, Nevo--Petersen's conjecture holds for any flag unimodular triangulation of the boundary $\partial\Bc_G$ in this case.
We remark that for any forest $G$, $\partial\Bc_G$ has a flag unimodular triangulation (Remark \ref{rem:B}).

The paper is organized as follows: In Section 2, we recall the definition of the $h^*$-polynomials of lattice polytopes, some properties of polynomials and their related conjectures.
In Section 3, we define the interior polynomials of connected bipartite graphs and recall a formula of the $\gamma$-polynomial of the $h^*$-polynomial of $\Ac_{\widehat{G}}$ for a graph $G$ in terms of interior polynomials. We give the proofs of Theorems \ref{formulaA}, \ref{thm:realrooted} and  \ref{thm:flag_A} in Sections 4, 5 and 6 respectively.
Finally, in Section~7, we extend the discussion to symmetric edge polytopes of type $B$.

\subsection*{Acknowledgements}
The authors would like to thank anonymous referees for reading the manuscript carefully.
The authors were partially supported by JSPS KAKENHI 18H01134, 19K14505 and 19J00312.


\section{Ehrhart theory and $\gamma$-polynomials}
\label{sec:Ehrhart}
In this section, we recall the definition of $h^*$-polynomials, the notion of $\gamma$-positivity and its related properties.
Let $\Pc \subset \RR^n$ be a lattice polytope of dimension $d$.
Given a positive integer $t$, we define
$$L_{\Pc}(t)=|t \Pc \cap \ZZ^n|,$$
where $t\Pc:=\{ t \xb \in \RR^n : \xb \in \Pc\}$.
The study on $L_{\Pc}(t)$ originated in Ehrhart \cite{Ehrhart} who proved that $L_{\Pc}(t)$ is a polynomial in $t$ of degree $d$ with the constant term $1$.
We call $L_{\Pc}(t)$ the \textit{Ehrhart polynomial} of $\Pc$.
The generating function of the lattice point enumerator, i.e., the formal power series
$$\text{Ehr}_\Pc(x)=1+\sum\limits_{k=1}^{\infty}L_{\Pc}(k)x^k$$
is called the \textit{Ehrhart series} of $\Pc$.
It is known that it can be expressed as a rational function of the form
$$\text{Ehr}_\Pc(x)=\frac{h^*(\Pc,x)}{(1-x)^{d+1}},$$
where $h^*(\Pc,x)$ is a polynomial in $x$ of degree at most $d$ with nonnegative integer coefficients \cite{Stanleynonnegative} and it
is called
the \textit{$h^*$-polynomial} (or the \textit{$\delta$-polynomial}) of $\Pc$. 
Moreover, $$h^*(\Pc,x)=\sum_{i=0}^{d} h_i^* x^i$$ satisfies $h^*_0=1$, $h^*_1=|\Pc \cap \ZZ^n|-(d +1)$ and $h^*_{d}=|{\rm relint} (\Pc) \cap \ZZ^n|$, where ${\rm relint} (\Pc)$ is the relative interior of $\Pc$.
Furthermore, $h^*(\Pc,1)=\sum_{i=0}^{d} h_i^*$ is equal to the normalized volume of $\Pc$.
We refer the reader to \cite{BeckRobins} for the detailed information about Ehrhart polynomials and $h^*$-polynomials.

A full-dimensional lattice polytope $\Pc \subset \RR^n$ is called \textit{reflexive} if the origin of $\RR^n$ belongs to the interior of $\Pc$ and its dual polytope 
\[\Pc^\vee:=\{\yb \in \RR^n  :  \langle \xb,\yb \rangle \leq 1 \ \text{for all}\  \xb \in \Pc \}\]
is also a lattice polytope, where $\langle \xb,\yb \rangle$ is the usual inner product of $\RR^n$.
Two lattice polytopes  $\mathcal{P} \subset \mathbb{R}^n$ and $\mathcal{Q} \subset \mathbb{R}^{m}$ are said to be \textit{unimodularly equivalent} if there exists an affine map from the affine span ${\rm aff}(\mathcal{P})$ of $\mathcal{P}$ to the affine span ${\rm aff}(\mathcal{Q})$ of $\mathcal{Q}$ that maps $\mathbb{Z}^n \cap {\rm aff}(\mathcal{P})$ bijectively onto $\mathbb{Z}^m \cap {\rm aff}(\mathcal{Q})$ and maps $\mathcal{P}$ to $\mathcal{Q}$.
Each lattice polytope is unimodularly equivalent to a full-dimensional lattice polytope.
In general, we say that a lattice polytope is reflexive if it is unimodularly equivalent to a reflexive polytope.
We can characterize when a lattice polytope is reflexive in terms of its $h^*$-polynomial.
A polynomial $f= \sum_{i=0}^{d}a_i x^i \in \RR[x]$ of degree $d$ is called \textit{palindromic} if $a_i=a_{d-i}$ for any $i$.
\begin{proposition}[\cite{hibi}]
A lattice polytope $\Pc$ of dimension $d$ is reflexive if and only if $h^*(\Pc,x)$ is a palindromic polynomial of degree $d$.
\end{proposition}
Let  $f= \sum_{i=0}^{d}a_i x^i$ be a polynomial with real coefficients and $a_d \neq 0$.
We now focus on the following properties.
\begin{itemize}
	\item[(RR)] We say that $f$ is {\em real-rooted} if all its roots are real.  
	\item[(LC)] We say that $f$ is {\em log-concave} if $a_i^2 \geq a_{i-1}a_{i+1}$ for all $i$.
	\item[(UN)] We say that $f$ is {\em unimodal} if $a_0 \leq a_1 \leq \cdots \leq a_k \geq \cdots \geq a_d$ for some $k$.
\end{itemize}  
If all its coefficients are positive, then these properties satisfy the implications
\[
{\rm(RR)} \Rightarrow {\rm(LC)} \Rightarrow {\rm(UN)}.
\]
Assume that $f$ is palindromic. Then $f$ has a unique expression 
\[
f=\sum_{i \geq 0}^{\lfloor d/2\rfloor}
\gamma_i \  x^i (1+x)^{d-2i}
\]
with $\gamma_0,\gamma_1,\ldots,\gamma_{\lfloor d/2\rfloor} \in \RR$. 
The polynomial $\sum_{i \geq 0}\gamma_i \ x^i$ is called {\em $\gamma$-polynomial of $f$}.
We say that $f$ is {\em $\gamma$-positive} if each $\gamma_i \geq 0$.
 We can see that a $\gamma$-positive polynomial is real-rooted if and only if its $\gamma$-polynomial has only real roots \cite[Observation 4.2]{EulerianNumbers}.
 For a reflexive polytope $\Pc$, denote $\gamma(\Pc,x)$ the $\gamma$-polynomial of $h^*(\Pc,x)$.

For a given lattice polytope, a fundamental problem within the field of Ehrhart theory is to determine if its $h^*$-polynomial is unimodal.
One famous instance is given by reflexive polytopes that possess a regular unimodular triangulation.

\begin{proposition}[\cite{BR}]
	Let $\Pc$ be a reflexive polytope.
	If $\Pc$ possesses a regular unimodular triangulation, then $h^*(\Pc, x)$ is unimodal.
\end{proposition}

It is known that if a reflexive polytope possesses a flag regular unimodular triangulation all of whose maximal simplices contain the origin, then the $h^*$-polynomial coincides with the $h$-polynomial of a flag triangulation of a sphere \cite{BR}. 
For the $h$-polynomial of a flag triangulation of a sphere, Gal  conjectured the following:

\begin{conjecture}[\cite{Gal}]
The $h$-polynomial of any flag triangulation of a sphere is $\gamma$-positive.
\end{conjecture}

More strongly, Nevo and Petersen conjectured the following:
\begin{conjecture}[\cite{NP}]
\label{conj:NP}
The $\gamma$-polynomial of the $h$-polynomial of any flag triangulation of a sphere coincides with the $f$-polynomial of some balanced simplicial complex.
\end{conjecture}

\section{Interior polynomials}
In \cite{OTlocally}, for any graph $G$ a formula for $\gamma(\Ac_{\widehat{G}},x)$ in terms of interior polynomials was given. 
In this section, we recall the definition of interior polynomials and the formula.

A {\em hypergraph} is a pair $\Hc = (V, E)$, where $E=\{e_1,\ldots,e_n\}$ is a finite multiset of non-empty subsets of $V=\{v_1,\ldots,v_m\}$. 
Elements of $V$ are called vertices and the elements of $E$ are the  hyperedges.
Then we can associate $\Hc$ to a bipartite graph ${\rm Bip} \Hc$
on the vertex set $V \cup E$ with the edge set
$\{ \{v_i, e_j\} : v_i \in e_j\}.$
Assume that ${\rm Bip} \Hc$ is connected.
A {\em hypertree} in $\Hc$ is a function ${\bf f}: E \rightarrow \ZZ_{\ge 0}$
such that there exists a spanning tree $\Gamma$ of ${\rm Bip} \Hc$ 
whose vertices have degree ${\bf f} (e) +1$ at each $e \in E$.
Then we say that $\Gamma$ induces ${\bf f}$.
Let ${\rm HT}(\Hc)$ denote the set of all hypertrees in $\Hc$.
A hyperedge $e_j \in E$ is said to be {\em internally active}
with respect to the hypertree ${\bf f}$ if it is not possible to 
decrease ${\bf f}(e_j)$ by $1$ and increase ${\bf f}(e_{j'})$ ($j' < j$) by $1$
so that another hypertree results.
We call a hyperedge {\em internally inactive} with respect to a hypertree
if it is not internally active and denote the number of such hyperedges 
of ${\bf f}$ by $\overline{\iota} ({\bf f}) $.
Then the {\em interior polynomial} of $\Hc$
is the generating function 
$I_\Hc (x) = \sum_{{\bf f} \in {{\rm HT}(\Hc)}} x^{ \overline{\iota} ({\bf f}) }$.
It is known \cite[Proposition~6.1]{interior} that $\deg I_\Hc (x) \le \min\{|V|,|E|\} - 1$.
If $G = {\rm Bip} \Hc$, then we set ${\rm HT}(G) = {\rm HT}(\Hc)$ and $I_G (x) = I_\Hc (x)$.

Let $G$ be a finite graph on $[n]$ with the edge set $E(G)$.
Given a subset $S \subset [n]$,
\[
E_S:=\{e \in E(G) : |e \cap S | =1  \}
\] 
is called a \textit{cut} of $G$.
For example, we have $E_{\emptyset}=E_{[n]}= \emptyset$. In general, it follows that $E_S=E_{[n] \setminus S}$.
We identify $E_S$ with the subgraph of $G$ on the vertex set $[n]$ and the edge set $E_S$.
 By definition, $E_S$ is a bipartite graph. Let ${\rm Cut}(G)$ be the set of all cuts of $G$. Note that $|{\rm Cut}(G)|=2^{n-1}$.

Assume that $G$ is a bipartite graph with a bipartition
$V_1 \cup V_2 =[n]$.
Then let $\widetilde{G}$ be a connected bipartite graph on $[n+2]$
whose edge set is 
\[
E(\widetilde{G}) = E(G) \cup \{ \{i, n+1\}  : i \in V_1\} \cup \{ \{j, n+2\}  : j \in V_2 \cup \{n+1\}\}.
\]

\begin{theorem}[{\cite[Theorem 5.3]{OTlocally}}]
\label{thm:suspinterior}
Let $G$ be a finite graph on $[n]$. Then one has 
\[
\gamma(\Ac_{\widehat{G}},x)=\dfrac{1}{2^{n-1}} \sum_{H \in {\rm Cut}(G)} I_{\widetilde{H}}(4x).
\]
In particular, ${\rm Vol}(\Ac_{\widehat{G}})=2 \sum_{H \in {\rm Cut}(G)} |{\rm HT}(\widetilde{H})|.$
\end{theorem}

\section{A formula of  $\gamma(\Ac_{\widehat{G}},x)$ for a cactus graph $G$}
In this section, we prove Theorem \ref{formulaA}.
First we recall a relation between interior polynomials and $k$-matchings.
Let $G$ be a finite graph with $n$ vertices.
A {\em $k$-matching} of $G$ is a set of $k$ pairwise non-adjacent edges of $G$.
Let
$$
M(G, k)=
\left\{\{v_{i_1},\ldots , v_{i_{2k}}\}
:
\begin{array}{c}
\mbox{there exists a } k\mbox{-matching of }
G\\
 \mbox{ whose vertex set is }
\{v_{i_1},\ldots , v_{i_{2k}}
\}
\end{array}
\right\}.
$$
For $k=0$, we set $M(G,0) = \{\emptyset\}$.

\begin{proposition}[{\cite[Proposition 3.4]{OTinterior}}]
\label{prop:matchinginterior}
Let $G$ be a bipartite graph. Then we have
\[
I_{\widetilde{G}}(x)=\sum_{k \geq 0}|M(G,k)| x^k.
\]
\end{proposition}

The {\em matching  polynomial} $\alpha(G,x)$ of $G$ is
$$
\alpha(G,x) = \sum_{k \ge 0} (-1)^k m_k ( G ) x^{n-2k},
$$
where $m_k(G)$ is the number of $k$-matchings in $G$.
On the other hand, 
the {\em matching generating polynomial} $g(G,x)$ of $G$ is
$$g(G,x) = \sum_{k \ge 0} m_k ( G ) x^k.$$
It is known \cite[Theorem 5.5.1]{Gut} that $\alpha(G,x)$ is real-rooted
 for any graph $G$.
Since $\alpha(G,x) = x^n g(G,-x^{-2})$, it follows that 
any root of $g(G,x)$ is real and negative.
In fact, if $u$ is a root of $g(G,x)$, then $u$ is not zero and 
$(-u)^{-1/2}$ is a root of $\alpha(G,x)$.
Thus $v=(-u)^{-1/2}$ is real and hence $u = -v^{-2}$ is real and negative.


\begin{lemma}
\label{klemma}
Let $G$ be a graph such that each edge of $G$ belongs to at most one even cycle of $G$.
Then
$$
|M(G, k)| =m_k(G)+
\sum_{R \in {\mathcal R}_2' (G)} (-1)^{c(R)} m_{k-|E(R)|/2} \left(G- R \right),
$$
where 
${\mathcal R}_2' (G)$ is the set of all subgraphs of $G$ 
consisting of vertex-disjoint even cycles,
and $c(R)$ is the number of the cycles of $R$.
\end{lemma}

\begin{proof}
Let $V=\{v_{i_1},\ldots , v_{i_{2k}}\}$ be an element of $M(G,k)$
and let $M(V)$ be the set of all $k$-matchings of $G$ whose vertex set is $V$.
Given $M, M' \in M(V)$, let $G_{M, M'}$ be the subgraph of $G$ whose edge set is $(M \cup M') \setminus (M \cap M')$.
It then follows that $G_{M, M'}$ is a regular bipartite graph of degree 2,
and hence belongs to ${\mathcal R}_2' (G)$.
Let ${\rm EC}(V)$ be the set of all even cycles $C$ of $G$
satisfying that there exist $M_1 ,M_2 \in M(V)$ such that $G_{M_1, M_2}$ 
contains $C$.

\bigskip

\noindent{\bf Claim 1.}
The set ${\rm EC}(V)$ consists of pairwise vertex-disjoint even cycles.

Suppose that distinct even cycles $C, C' \in {\rm EC}(V)$ have a common vertex $v$.
Then there exists  $M_1,M_2,M_3,M_4 \in M(V)$ 
such that $C$ appears in $G_{M_1, M_2}$ and $C'$ appears in $G_{M_3, M_4}$.
Since  $C$ appears in $G_{M_1, M_2}$, $M_1 \cap C$ has an edge of the form $\{v, w\}$.
Similarly, since $C'$ appears in $G_{M_3, M_4}$, $M_3 \cap C'$ has an edge of the form $\{v, w'\}$.
Since $C$ and $C'$ have no common edges, $\{v, w\}$ (resp. $\{v, w'\}$) is not an edge of $C'$ (resp. $C$).
It then follows that both $\{v, w\}$ and $\{v, w'\}$ appear in $G_{M_1, M_3}$.
Hence there exists an even cycle $C''$ of $G$ that contains $\{v, w\}$ and $\{v, w'\}$.
This contradicts to the hypothesis that 
 each edge of $G$ belongs to at most one even cycle of $G$.
Thus ${\rm EC}(V)$ consists of pairwise vertex-disjoint even cycles.

\bigskip

\noindent{\bf Claim 2.}
If $M \in M(V)$ and $C=(\{v_1,v_2\},\ldots,\{v_{2\ell-1},v_{2\ell}\},\{ v_{2\ell}, v_1\}) \in {\rm EC}(V)$,
then $M \cap C$ is either $\{  \{v_1,v_2\},\{v_3,v_4\},\ldots,\{v_{2\ell-1},v_{2\ell}\} \}$
or 
$\{ \{v_2,v_3\},\{v_4,v_5\},\ldots,\{ v_{2\ell}, v_1\} \}$.

Since $C$ belongs to ${\rm EC}(V)$, there exists $M_1, M_2 \in M(V)$ such that
$C$ appears in the graph $G_{M_1,M_2}$.
Suppose that there exists $j$ such that neither $\{v_{j-1},v_j\}$ nor $\{v_j, v_{j+1}\}$ belongs to $M$.
Then $M$ has an edge $\{v_j , w\}$ with $w \notin \{v_{j-1}, v_{j+1}\}$.
We may assume that $\{v_{j-1},v_j\}$ belongs to $M_1$.
Then $G_{M,M_1}$ has edges $\{v_{j-1},v_j\}$ and $\{v_j , w\}$.
Hence there exists an even cycle of $G$ that contains $\{v_{j-1},v_j\}$ and $\{v_j , w\}$,
a contradiction.
Thus one of $\{v_{j-1},v_j\}$ or $\{v_j, v_{j+1}\}$ belongs to $M$.
Hence the intersection of $C$ and $M$ is either 
$\{  \{v_1,v_2\},\{v_3,v_4\},\ldots,\{v_{2\ell-1},v_{2\ell}\} \}$
or 
$\{ \{v_2,v_3\},\{v_4,v_5\},\ldots,\{ v_{2\ell}, v_1\} \}$.

\bigskip

\noindent{\bf Claim 3.}
If an edge $e$ does not appear in any cycle in ${\rm EC}(V)$,
then either $e \in M$ for all $M \in M(V)$ or  $e \notin M$ for all $M \in M(V)$.

Suppose that $e$ belongs to $M \in M(V)$ and does not belong to $M' \in M(V)$.
Then $e$ belongs to $(M \cup M') \setminus (M \cap M')$ and hence
appears in the graph $G_{M,M'}$.
Thus there exists an even cycle $C \in {\rm EC}(V)$ such that $e \in C$.

\bigskip

Let $r =r(V)$ denote the number of cycles in $ {\rm EC}(V) $ and let $ {\rm EC}(V) = \{ C_1,\ldots, C_r \}$.
Let
$E'$ be the set of all edges $e$ of $G$ such that 
$e \in M$ for all $M \in M(V)$.
From Claims 1--3,
$$
M(V) = \{
E' \cup M_1 \cup \cdots \cup M_r : 
M_i \mbox{ is one of two perfect matchings of } C_i
\}.
$$
Thus we have $|M(V)| =2^r$.
On the other hand, for $R=C_{i_1} \cup \cdots \cup C_{i_s} \in {\mathcal R}_2' (G)$
with $\{C_{i_1},\ldots,C_{i_s}\} \subset {\rm EC}(V)$, 
the number of $(k-|E(R)|/2)$-matchings of $G-R$ with vertex set $V \setminus V(R)$ is $ |M(V \setminus V(R))| = 2^{r-s}$.
Since
$$
2^r + \sum_{s=1}^{r} (-1)^s \binom{r}{s} 2^{r-s}
=1,
$$
we have 
\begin{eqnarray*}
|M(G,k)| &=& \sum_{V \in M(G,k)} 1\\
&=&\left( \sum_{V \in M(G,k)} 2^{r(V)} \right)
+ \left( \sum_{V \in M(G,k)} \sum_{s=1}^{r(V)} (-1)^s \binom{r(V)}{s} 2^{r(V)-s} \right)\\
&=& \left(\sum_{V \in M(G,k)} |M(V)| \right)+\left(\sum_{V \in M(G,k)} \sum_{
\substack{R \in {\mathcal R}_2' (G)\\R \subset\bigcup_{C \in {\rm EC}(V)} C}} (-1)^{c(R)} |M(V \setminus V(R))| \right)\\
&=&\left( \sum_{V \in M(G,k)} |M(V)|\right) +
\left(\sum_{R \in {\mathcal R}_2' (G)} (-1)^{c(R)}
\sum_{
\substack{
{V \in M(G,k)}\\\bigcup_{C \in {\rm EC}(V)} C \supset R
}
}
 |M(V \setminus V(R))| \right)\\
&=& m_k(G) +\sum_{R \in {\mathcal R}_2' (G)} (-1)^{c(R)} m_{k-|E(R)|/2} \left(G- R \right),
\end{eqnarray*}
as desired.
\end{proof}

Now, we prove Theorem \ref{formulaA}.
In fact, Theorem \ref{formulaA} follows from the following more general result.
\begin{theorem}
\label{formulaAgen}
Let $G$ be a graph such that each edge of $G$ belongs to at most one even cycle of $G$.
Then one has
$$
\gamma(\Ac_{\widehat{G}},x)=g(G,2x) + \sum_{R \in {{\mathcal R}_2'}(G)} (-2)^{c(R)} g (G-R , 2x) \ 
x^{\frac{|E(R)|}{2} }.
$$
	Moreover, the normalized volume of ${\mathcal A}_{\widehat{G}}$ is 
$$
2^n g (G,  1/2) + \sum_{R \in {{\mathcal R}_2'}(G)} (-2)^{c(R)} 2^{n-|E(R)|} g (G-R ,1/2).
$$
\end{theorem}

\begin{proof}
Let $n$ be the number of vertices of $G$.
From Theorem \ref{thm:suspinterior} and Proposition \ref{prop:matchinginterior}
one has
$$
\gamma(\Ac_{\widehat{G}},x)=\frac{1}{2^{n-1}} \sum_{H \in {\rm Cut}(G)} I_{\widetilde{H}} (4x)
=  \frac{1}{2^{n-1}} \sum_{H \in {\rm Cut}(G)} \sum_{k \ge 0} |M(H,k)|\  (4x)^k.
$$
Since each edge of $G$ belongs to at most one even cycle of $G$,
each $H \in {\rm Cut}(G)$ satisfies the same condition.
From Lemma~\ref{klemma}, 
$$
|M(H, k)| =m_k(H)+
\sum_{R \in {\mathcal R}_2' (H)} (-1)^{c(R)} m_{k-|E(R)|/2} \left(H- R \right),
$$
for each $H \in {\rm Cut}(G)$.
Thus the $\gamma$-polynomial of ${\mathcal A}_{\widehat{G}}$ is 
\begin{eqnarray*}
 & & 
\frac{1}{2^{n-1}} \sum_{H \in {\rm Cut}(G)} \sum_{k \ge 0}  m_k(H) \  (4x)^k\\
&+&
\frac{1}{2^{n-1}} \sum_{H \in {\rm Cut}(G)} \sum_{k \ge 0} \sum_{R \in {\mathcal R}_2' (H)} (-1)^{c(R)} m_{k-|E(R)|/2} \left(H- R \right) \  (4x)^k.
\end{eqnarray*}
Note that every $k$-matching of $H \in {\rm Cut}(G)$
is a $k$-matching of $G$.
\begin{itemize}
\item
Let $M$ be a $k$-matching of $G$.
Then $M$ is a $k$-matching of $H \in {\rm Cut}(G)$
if and only if $M$ is a subgraph of $H$.
There are $2^{n-k-1}$ such $H \in {\rm Cut}(G)$.
\item
Let $M$ be a $(k-|E(R)|/2)$-matching of $G-R$ with $R \in {\mathcal R}_2' (G)$.
Then $M$ is a matching of $H \in {\rm Cut}(G)$ with $R \in {\mathcal R}_2' (H)$
if and only if $M  \cup R$ is a subgraph of $H$.
There are 
$$
2^{n - |E(R)|-1- (k-|E(R)|/2) +c(R)} 
=
2^{n-k-1-|E(R)|/2 +c(R)} 
$$
cuts $H \in {\rm Cut}(G)$
such that $M \cup R$ is a subgraph of $H$.
\end{itemize}
Thus $\gamma(\Ac_{\widehat{G}},x)$ is equal to 
\begin{eqnarray*}
 & & 
\frac{1}{2^{n-1}} \sum_{k \ge 0} 2^{n-k-1} m_k(G) \  (4x)^k\\
&+&
\frac{1}{2^{n-1}} \sum_{k \ge 0} \sum_{R \in {\mathcal R}_2' (G)} 2^{n-k-1-\frac{|E(R)|}{2}+c(R)} (-1)^{c(R)} m_{k-|E(R)|/2} \left(G- R \right) \  (4x)^k\\
&=&\sum_{k \ge 0} m_k(G) \  (2x)^k +
\sum_{R \in {\mathcal R}_2' (G)} (-2)^{c(R)} \sum_{k \ge 0} 
m_{k-|E(R)|/2} \left(G- R \right) \  (2x)^{k-\frac{|E(R)|}{2}} x^{\frac{|E(R)|}{2}}\\
&=& g(G, 2x) +
\sum_{R \in {{\mathcal R}_2'}(G)} (-2)^{c(R)} g (G-R , 2x) \ 
x^{\frac{|E(R)|}{2} }.
\end{eqnarray*}
Moreover, since $h^*({\mathcal A}_{\widehat{G}},x)
=
(x+1)^n \gamma(\Ac_{\widehat{G}},x/(x+1)^2)$,
the normalized volume of ${\mathcal A}_{\widehat{G}}$ is 
$$
h^*({\mathcal A}_{\widehat{G}},1)
=
2^n \gamma(\Ac_{\widehat{G}},1/4)
=
2^n g (G,  1/2) + \sum_{R \in {{\mathcal R}_2'}(G)} (-2)^{c(R)} 2^{n-|E(R)|} g (G-R ,1/2).
$$
\end{proof}

Since every matching generating polynomial is real-rooted, we obtain the following.
\begin{corollary}
\label{cor:noeven}
Let $G$ be a finite graph with $n$ vertices.
If $G$ has no even cycles, then the $\gamma$-polynomial of the $h^*$-polynomial of 
${\mathcal A}_{\widehat{G}}$ is $g(G,2x)$.
	In particular, $	h^*({\mathcal A}_{\widehat{G}}, x)$ is real-rooted.
	Moreover, the normalized volume of ${\mathcal A}_{\widehat{G}}$ is $2^n g(G,1/2)$.
\end{corollary}

An example of the suspension of a cactus graph is a wheel graph. In \cite{DDM}, the normalized volume of the symmetric edge polytope of a wheel graph was computed. By using Theorem~\ref{formulaA}, we compute the $\gamma$-polynomial of the $h^*$-polynomial and the normalized volume  of the polytope.

\begin{example}
\label{wheelexample}
Let $C_n$ be a cycle of length $n$.
Then $\widehat{C_n}$ is a wheel graph.
It is known that 
$$g(C_n, x)=L_n(x),$$ 
where
$L_n(x)$ is the {\em Lucas polynomial} defined by
$$
L_0(x) =2, \ L_1(x) = 1, \ L_d(x) = L_{d-1}(x)+xL_{d-2}(x),
$$
and 
\begin{eqnarray*}
g(C_n,2x) &=& L_n(2x) = \frac{
(1+\sqrt{1+8x})^n + (1-\sqrt{1+8x})^n
}{2^n}.
\end{eqnarray*}
See, e.g., \cite[p.27 and p.36]{SW}.
From Theorem \ref{formulaA}
one has
$$
\gamma(\Ac_{\widehat{C_n}},x)=
\left\{
\begin{array}{ll}
 \frac{
(1+\sqrt{1+8x})^n + (1-\sqrt{1+8x})^n
}{2^n} & \mbox{if } n \mbox{ is odd,}\\
\\
 \frac{
(1+\sqrt{1+8x})^n + (1-\sqrt{1+8x})^n
}{2^n}  -2 x^{\frac{n}{2} } & \mbox{otherwise}.
\end{array}
\right.
$$
In particular, we obtain $$
\Vol(\Ac_{\widehat{C_n}})=
2^n \gamma(\Ac_{\widehat{C_n}},1/4)
=
 \left\{
\begin{array}{ll}
(1+\sqrt{3})^n + (1-\sqrt{3})^n & \mbox{if } n \mbox{ is odd,}\\
\\
(1+\sqrt{3})^n + (1-\sqrt{3})^n -2& \mbox{otherwise}.
\end{array}
\right.
$$
This coincides with \cite[Theorem 4.24]{DDM}.
\end{example}

\section{Real-rootedness of $h^*(\Ac_{\widehat{G}},x)$}
In this section, we prove Theorem \ref{thm:realrooted}.
The polynomial
$$
\gamma(G, x):=
g(G,2x) + \sum_{R \in {{\mathcal R}_2'}(G)} (-2)^{c(R)} g (G-R , 2x) \ 
x^{\frac{|E(R)|}{2} }
$$
is strongly related with 
the {\em $\mu$-polynomial} in methods from  chemical graph theory.
Suppose that $G$ has $n$ vertices and $r$ cycles $C_1, \ldots, C_r$.
Let ${\bf t}=(t_1,\ldots, t_r) \in \RR^r$ be a vector whose 
component $t_i$ is associated to the cycle $C_i$ for $i=1,2,\ldots,r$.
It is known \cite[Proposition 1a]{GP} that
the  $\mu$-polynomial $\mu(G, {\bf t}, x) $ of a graph $G$ satisfies
$$
\mu(G, {\bf t}, x) = \alpha(G,x) +
\sum_{R \in {{\mathcal R}_2}(G)} (-2)^{c(R)} \alpha(G-R, x) \prod_{C_i \subset R} t_i,
$$
where 
${\mathcal R}_2 (G)$ is the set of all subgraphs of $G$ 
consisting of vertex-disjoint cycles,
$c(R)$ is the number of the cycles of $R$.
This polynomial generalizes important graph polynomials.
In fact, we have $\mu(G, {\bf 0}, x) = \alpha(G,x)$ and $\mu(G, {\bf 1}, x) = \phi(G,x)$,
where $ \phi(G,x)$ is the {\em characteristic polynomial} of $G$.
See \cite[Theorem 5.3.3]{Gut}.
For a cactus graph, the following is known.

\begin{proposition}[{\cite[Proposition 5]{GP}}]
\label{cactusreal}
Let $G$ be a cactus graph.
Then $\mu(G, {\bf t}, x) $ is real-rooted if $|t_i| \le 1$ for all $1 \le i \le r$.
\end{proposition}

Thus we can prove Theorem \ref{thm:realrooted}.

\begin{proof}[Proof of Theorem \ref{thm:realrooted}]
It is enough to show that $\gamma(G,x)$ is real-rooted.
The polynomial $\gamma(G, x)$ satisfies
\begin{eqnarray*}
x^n \gamma \left(G, -\frac{1}{2x^2} \right)
& = &
x^n g(G,-x^{-2} ) + \sum_{R \in {{\mathcal R}_2'}(G)} (-2)^{c(R)} x^n g (G-R , -x^{-2}) \ 
\left(-\frac{1}{2x^2}\right)^{\frac{|E(R)|}{2} }\\
& = &
\alpha(G,x) +  \sum_{R \in {{\mathcal R}_2'}(G)} (-2)^{c(R)} \alpha (G-R , x) \ 
\prod_{C_i \subset R} \left(-\frac{1}{2}\right)^{\frac{|E(C_i)|}{2} }\\
&=&
\mu(G, {\bf t}, x) ,
\end{eqnarray*}
where 
${\bf t}=(t_1,\ldots,t_r)$ with
$$
t_i=
\left\{
\begin{array}{cc}
\left(-\frac{1}{2}\right)^{\frac{|E(C_i)|}{2} } & \mbox{ if }  C_i \mbox{ is an even cycle},\\
\\
0 & \mbox{otherwise.}
\end{array}
\right.
$$
By Proposition~\ref{cactusreal}, this is real-rooted.
If $u$ is a root of $\gamma(G,x)$, then $u \neq 0$ and
$x=1/\sqrt{-2u}$ is a root of $\mu(G, {\bf t}, x)$.
Since $1/\sqrt{-2u}$ is real, so is $u$.
\end{proof}
From Theorem \ref{thm:realrooted} the $h^*$-polynomial of the symmetric edge polytope of a wheel graph, i.e., the suspension of a cycle is real-rooted. However, the $h^*$-polynomial of the symmetric edge polytope of a cycle is not always real-rooted.

\begin{example}
\label{ex:nonreal}
	Let $C_n$ be a cycle of length $n \geq 3$.
	Then from \cite[Proposition 5.7]{OTlocally} one has 
	\[
	\gamma(\Ac_{C_n},x)=\sum_{i=0}^{\left\lfloor \frac{n-1}{2} \right\rfloor} \binom{2i}{i} x^i.
	\]
	Hence $h^*(\Ac_{C_n},x)$ is $\gamma$-positive. 
	However, when $n=5$, the $\gamma$-polynomial $\gamma(\Ac_{C_5},x)=1+2x+6x^2$ is not real-rooted. Hence $h^*(\Ac_{C_5},x)$ is not real-rooted.
\end{example}

\section{Nevo--Petersen Conjecture for $h^*(\Ac_{\widehat{G}},x)$}
An $n$-dimensional simplicial complex $\Delta$ is said to be
\begin{itemize}
	\item \textit{flag} if all minimal non-faces of $\Delta$ contain only two elements;
	\item \textit{balanced} if there is a proper coloring of its vertices $c : V \to [n+1]$, where $V$ is the vertex set of $\Delta$. 
\end{itemize}
Frohmader showed that the $f$-vector of a flag simplicial complex coincides with that of a balanced simplicial complex.
Nevo and Petersen posed the following strengthening problem of Conjecture \ref{conj:NP}.

\begin{problem}[\cite{NP}]
The $\gamma$-polynomial of the $h$-polynomial of any flag triangulation of a sphere coincides with the $f$-polynomial of some flag simplicial complex.	
\end{problem}

In this section, we discuss this problem for $h^*(\Ac_{\widehat{G}},x)$. 
In particular, we prove Theorem~\ref{thm:flag_A}.
First, we recall that every flag simplicial complex arises from a finite simple graph. Let $G$ be a finite simple graph on $[n]$ with the edge set $E(G)$.
A subset $C$ of $[n]$ is called a \textit{clique} of $G$ if for all $i$ and $j$ belonging to $C$ with $i \neq j$,  one has $\{i,j\} \in E(G)$. The \textit{clique complex} of $G$ is the simplicial complex $\Delta(G)$ on $[n]$ whose faces are the cliques of $G$.
\begin{lemma}[{\cite[Lemma 9.1.3]{monomial}}]
	A simplicial complex $\Delta$ is flag if and only if $\Delta$ is the clique complex of a finite simple graph.
\end{lemma}
Hence the $f$-polynomial of a flag simplicial complex can be computed by counting cliques of a graph. By considering the complement of a graph, we also can compute the $f$-polynomial of a flag simplicial complex by counting independent sets of a graph.
Let us denote by $\overline{G}$ the complement of $G$. A subset $S$ of $[n]$ is called an \textit{independent set} if for all $i$ and $j$ belonging to $S$ with $i \neq j$,  one has $\{i,j\} \notin E(G)$.  
Let $i_k$ denote the number of independent sets $S$ of $G$ such that $|S|=k$.
(We set $i_0=1$.)
The \textit{independence polynomial} of $G$ is
\[
i(G,x)=\sum_{k \geq 0} i_k x^k
.\]
Since a subset $S$ of $[n]$ is a clique of $G$ if and only if $S$ is an independent set of $\overline{G}$, 
the $f$-polynomial of $\Delta(G)$ is equal to $i(\overline{G},x)$.
By using this correspondence, we can prove that any matching generating polynomial coincides with the $f$-polynomial of a flag simplicial complex.
For a graph $G$ with the edge set $E(G)$, the \textit{line graph} $L(G)$ is the graph with vertex set $E(G)$ and such that vertices $e,f \in E(G)$ with $e \neq f$ are adjacent if and only if $e \cap f \neq \emptyset$.
\begin{proposition}
	\label{prop:matchflag}
	Let $G$ be a finite simple graph. Then the matching generating polynomial $g(G,x)$ of $G$ coincides with the $f$-polynomial of a flag simplicial complex.
\end{proposition}
\begin{proof}
 It follows that the matching generating polynomial $g(G,x)$ of $G$ and the independence polynomial $i(L(G),x)$ of the line graph $L(G)$ of $G$ are identical. Namely, one has $g(G,x)=i(L(G),x)$. 
Hence $g(G,x)$ is the $f$-polynomial of a flag simplicial complex
that is the clique complex of $\overline{L(G)}$.
\end{proof}

Moreover, by using the above correspondence, we prove the following proposition.

\begin{proposition}
	\label{fmx}
Suppose that $f(x)$ is the $f$-polynomial of a flag simplicial complex.
Then, for any $0 < m \in {\mathbb Z}$, 
there exists a flag simplicial complex whose $f$-polynomial is $f(mx)$.
\end{proposition}

In order to show this proposition, the following lemma is needed.
For two graphs $G$ and $H$ with the vertex sets $V(G)$ and $V(H)$ respectively, let $G[H]$ be the graph with the vertex set $V(G) \times V(H)$ and such that a vertex $(a,x)$ is adjacent to a vertex $(b,y)$ if and only if $a$ is adjacent to $b$ (in $G$) or $a=b$ and $x$ is adjacent to $y$ (in $H$).
The graph $G[H]$ is called \textit{lexicographic product} (or \textit{composition}) of $G$ and $H$.

\begin{lemma}[{\cite[Theorem 1]{BHN}}]
\label{lem:independence}
Let $G$ and $H$ be graphs. Then one has
\[
i(G[H],x)=i(G,i(H,x)-1).
\]	
\end{lemma}

\begin{proof}[Proof of Proposition \ref{fmx}]
Let $\Delta$ be a flag simplicial complex whose $f$-polynomial is $f(x)$.
Then there exists a graph $G$ such that $\Delta$ 
is the clique complex of $G$.
Moreover,  $f(x)$ is the independence polynomial $i(\overline{G},x)$ of $\overline{G}$.
We consider the lexicographic product  $\overline{G}[K_m]$ of $\overline{G}$ and a complete graph $K_m$ of $m$ vertices.
It then follows from Lemma \ref{lem:independence} that
\[i(\overline{G}[K_m] ,x) = i(\overline{G}, i(K_m,x)-1)= i(\overline{G}, mx)=f(mx).\] 
Thus $f(mx)$ is the $f$-polynomial of a flag simplicial complex
that is the clique complex of the complement graph of $\overline{G}[K_m]$.
\end{proof}

Now, we prove Theorem \ref{thm:flag_A}.

\begin{proof}[Proof of Theorem \ref{thm:flag_A}]
Let $G$ be a graph which has no even cycles. 
From Corollary \ref{cor:noeven},  one has \[\gamma({\mathcal A}_{\widehat{G}},x)=g(G,2x).\]
Thus Propositions~\ref{prop:matchflag} and \ref{fmx} guarantee that $g(G,2x) $
is the $f$-polynomial of a flag simplicial complex.
\end{proof}

\begin{example}
Let $C_n$ be an odd cycle of length $n$.
It then follows from Example \ref{wheelexample} that
\[
\gamma(\Ac_{\widehat{C_n}},x)=g(C_n, 2x)=
 \frac{
(1+\sqrt{1+8x})^n + (1-\sqrt{1+8x})^n
}{2^n}.
\]
From Theorem~\ref{thm:flag_A} $\gamma(\Ac_{\widehat{C_n}},x)$ is the $f$-polynomial of a flag simplicial complex.
In fact, the line graph of a cycle is isomorphic to itself. Hence $\gamma(\Ac_{\widehat{C_n}},x)$ is the $f$-polynomial of the clique complex of $\overline{C_n[K_2]}$.
For example, if $n=3$, then $\overline{C_3[K_2]}=\overline{K_3[K_2]}=\overline{K_6}$ is the empty graph with $6$ vertices and
$\gamma(\Ac_{\widehat{C_3}},x)=g(C_3, 2x)=1+6x$
coincides with the $f$-polynomial of its clique complex.
\end{example}

In the rest of this section, we show that for any bipartite graph $G$ such that every cycle of length $\geq 6$ in $G$ has a chord, 
$\partial\Ac_{\widehat{G}}$ has a flag unimodular triangulation. 
In particular, for any forest $G$, 
$\partial\Ac_{\widehat{G}}$ has a flag unimodular triangulation. 
First, we introduce the theory of Gr\"obner bases of toric ideals.
(See, e.g., \cite[Chapter~3]{binomialideals} for details on toric ideals and Gr\"obner bases.)
Let $\Pc \subset \RR^n$ be a lattice polytope,
where $\Pc \cap \ZZ^n =\{\ab_1,\ldots,\ab_m\}$.
For simplicity, we assume that $\Pc$ is {\em spanning}, i.e.,
$\ZZ^{n+1} = \sum_{i=1}^m \ZZ (\ab_i,1)$. 
Note that for any graph $G$, $\Ac_G$ is unimodularly equivalent to a full-dimensional lattice polytope satisfying this condition.
Let 
$$\Rc=K[t_1, t_1^{-1}, \dots, t_n, t_n^{-1} ,s ]$$
be the Laurent polynomial ring over a field $K$
and let 
$$\Sc=K[x_1, \dots, x_m]$$
be the polynomial ring over $K$.
We define the ring homomorphism $ \pi : \Sc \rightarrow \Rc$ by setting $\pi(x_i) = 
t_1^{a_{i 1}} \cdots t_n^{a_{i n}}  s$ where $\ab_i=(a_{i 1},\ldots,a_{i n})$.
The {\em toric ideal} $I_{\Pc}$ of $\Pc$ is the kernel of $\pi$.
It is known that $I_{\Pc}$ is generated by homogeneous binomials.
Given a monomial order $<$, the {\em initial ideal}  ${\rm in}_<(I_{\Pc})$ of $I_{\Pc}$  with respect to $<$
is an ideal generated by the initial monomials ${\rm in}_<(f)$ of nonzero polynomials $f$ in $I_{\Pc}$.
The {\em initial complex} $\Delta(\Pc,<)$ of $\Pc$ with respect to $<$ is 
\[
\Delta(\Pc,<)=
\left\{
\conv (B) : B \subset \{\ab_1,\ldots,\ab_m\}, \prod_{\ab_i \in B} x_i \notin 
\sqrt{ {\rm in}_<(I_{\Pc}) }
\right\},
\]
where $\sqrt{ {\rm in}_<(I_{\Pc}) }$ is the radical of ${\rm in}_<(I_{\Pc})$.

\begin{proposition}[{\cite[Theorems 4.14 and 4.17]{binomialideals}}]
\label{stu1}
The initial complex $\Delta(\Pc,<)$ is a triangulation of $\Pc$.
Moreover, $\Delta(\Pc,<)$ is flag unimodular if and only if 
${\rm in}_<(I_{\Pc})$ is generated by squarefree quadratic monomials.
\end{proposition}

We have the following proposition from a fact {\cite[Proposition~8.6]{Stu}} 
on the initial complex with respect to a reverse lexicographic order.

\begin{proposition}
\label{stu2}
Suppose that $\ab_1={\bf 0}$ is the unique lattice point
in the interior of $\Pc$.
Let $<$ be a reverse lexicographic order such that the smallest variable
is $x_1$.
Then ${\bf 0}$ is a vertex of every maximal simplex in $\Delta(\Pc,<)$,
and
\[
\Delta=
\left\{
\conv (B) : B \subset \{\ab_2,\ldots,\ab_m\}, \prod_{\ab_i \in B} x_i \notin 
\sqrt{ {\rm in}_<(I_{\Pc}) }
\right\}
\]
is a triangulation of the boundary $\partial \Pc$ of $\Pc$.
In particular, if $\Delta(\Pc,<)$ is flag and unimodular, then so is $\Delta$.
\end{proposition}

Given a graph $G$ on the vertex set $[n]$ and the edge set $E(G)=\{e_1 ,\dots, e_m\}$,
let 
$$\Rc=K[t_1, t_1^{-1}, \dots, t_n, t_n^{-1} ,s ]$$
be the Laurent polynomial ring over a field $K$
and let 
$$\Sc=K[x_1, \dots, x_m, y_1, \dots, y_m, z]$$
be the polynomial ring over $K$.
We define the ring homomorphism $ \pi : \Sc \rightarrow \Rc$ by setting $\pi(z) = s$, 
$\pi(x_k) = t_i t_j^{-1} s$ and $\pi(y_k) = t_i^{-1} t_j s$
if $e_k = \{i,j\} \in E(G)$ and $i < j$.
Then the toric ideal $I_{\Ac_G}$ of $\Ac_G$ is the kernel of $\pi$.
The initial ideal of $I_{\Ac_G}$ plays an important role in, e.g., 
\cite{DDM, HJMsymmetric, OS, OTlocally}.
In particular, a Gr\"obner basis of $I_{\Ac_G}$ of a graph $G$
with respect to a certain 
reverse lexicographic order is given in \cite[Proposition 3.8]{HJMsymmetric}.

\begin{theorem}
\label{flagtri}
Let $G$ be a bipartite graph such that every cycle of length $\ge 6$ in $G$ has a chord.
Then there exists a reverse lexicographic order $<$ such that 
\begin{itemize}
\item[{\rm (i)}]
$z$ is the smallest variable with respect to $<${\rm ;}
\item[{\rm (ii)}]
${\rm in}_< (I_{\Ac_{\widehat{G}}})$ is generated by squarefree quadratic monomials.
\end{itemize}
In particular, $\partial\Ac_{\widehat{G}}$ has a flag unimodular triangulation.
\end{theorem}

\begin{proof}
Let $[n]$ be the vertex set of $G$.
Recall that  $\widetilde{G}$ is a bipartite graph on $[n+2]$
whose edge set is 
\[
E(\widetilde{G}) = E(G) \cup \{ \{i, n+1\}  : i \in V_1\} \cup \{ \{j, n+2\}  : j \in V_2 \} \cup \{\{n+1, n+2\}\}.
\]
Since every cycle of length $\ge 6$ in $G$ has a chord,
 every cycle of length $\ge 6$ in $\widetilde{G}$ has a chord.
From \cite[Theorem~4.4]{CSC}, 
there exists a reverse lexicographic order $<$ such that 
$z$ is the smallest variable with respect to $<$, and that
the initial ideal of $I_{\Ac_{\widetilde{G}}}$ with respect to $<$ is generated by squarefree quadratic monomials $\{m_1,\ldots,m_s\}$.
Note that $\widehat{G}$ is obtained from $\widetilde{G}$ by contracting
the edge $\{n+1,n+2\}$, and 
there is a natural correspondence between $E(\widetilde{G}) \setminus \{\{n+1, n+2\}\}$ and $E(\widehat{G})$. 
From a fact shown in the proof of \cite[Proposition 5.4]{OTlocally}, 
the initial ideal of $I_{\Ac_{\widehat{G}}}$ with respect to 
the reverse lexicographic order induced by $<$ is
generated by squarefree quadratic monomials $\{m_1,\ldots,m_s\}\setminus
\{x_k y_k\}$ where $e_k=\{n+1,n+2\}$.
Thus $\partial\Ac_{\widehat{G}}$ has a flag unimodular triangulation
by Propositions \ref{stu1} and \ref{stu2}.
\end{proof}

\section{Symmetric edge polytopes of type B}
In this section, we consider the symmetric edge polytope  of type B of a cactus bipartite graph.
Note that the symmetric edge polytope $\Bc_G$ of a graph $G$ is reflexive if and only if $G$ is bipartite \cite[Theorem 0.1]{OTinterior}.
In the case, a formula of the $\gamma$-polynomial of $h^*(\Bc_G,x)$ in terms of interior polynomials is given.

\begin{theorem}[{\cite[Theorem 0.3]{OTinterior}}]
\label{interiorB}
	Let $G$ be a bipartite graph. Then one has
	\[
	\gamma(\Bc_G,x)=I_{\widetilde{G}}(4x).
	\]
\end{theorem}
Similarly to Theorem~\ref{formulaA}, for a cactus bipartite graph $G$, we give a formula of $\gamma(\Bc_G,x)$ in terms of matching generating polynomials and show that $h^*(\Bc_G,x)$ is real-rooted.

\begin{theorem}
\label{thm:cactusB}
Let $G$ be a cactus bipartite graph.
Then one has  
$$
\gamma({\mathcal B}_G,x) =
g(G,4x) + \sum_{R \in {{\mathcal R}_2'}(G)} (-1)^{c(R)} g (G-R , 4x) \ 
(4x)^{\frac{|E(R)|}{2} }.
$$
Moreover, $h^*(\Bc_G,x)$ is real-rooted.
\end{theorem}

\begin{proof}
Let $n$ be the number of vertices of $G$.
From Proposition \ref{prop:matchinginterior}, Lemma~\ref{klemma} and Theorem \ref{interiorB},
the $\gamma$-polynomial of ${\mathcal B}_G$ is 
\begin{eqnarray*}
I_{\widetilde{G}} (4x)
&=& \sum_{k \ge 0} |M(G,k)|\  (4x)^k\\
&=&\sum_{k \ge 0}
m_k(G) (4x)^k+\sum_{k \ge 0}
\sum_{R \in {\mathcal R}_2' (G)} (-1)^{c(R)} m_{k-|E(R)|/2} \left(G- R \right)(4x)^k\\
&=& g(G, 4x) + \sum_{R \in {\mathcal R}_2' (G)} (-1)^{c(R)} g(G-R, 4x) (4x)^{\frac{|E(R)|}{2} }.
\end{eqnarray*}
Moreover, $\gamma({\mathcal B}_G, x)$ satisfies
\begin{eqnarray*}
x^n \gamma \left({\mathcal B}_G , -\frac{1}{4x^2} \right)
& = &
x^n g(G,-x^{-2} ) + \sum_{R \in {{\mathcal R}_2'}(G)} (-1)^{c(R)} x^n g (G-R , -x^{-2}) \ 
\left(-\frac{1}{x^2}\right)^{\frac{|E(R)|}{2} }\\
& = &
\alpha(G,x) +  \sum_{R \in {{\mathcal R}_2'}(G)} (-2)^{c(R)} \alpha (G-R , x) \ 
\left(\frac{1}{2}\right)^{C(R)}
\prod_{C_i \subset R} (-1)^{\frac{|E(C_i)|}{2} }\\\\
&=&
\mu(G, {\bf t}, x) ,
\end{eqnarray*}
where 
${\bf t}=(t_1,\ldots,t_r)$ with
$t_i= \left(-1\right)^{\frac{|E(C_i)|}{2} } / 2$.
By Proposition~\ref{cactusreal}, this is real-rooted.
Hence $h^*(\Bc_G,x)$ is also real-rooted.
\end{proof}

By using Theorem \ref{thm:cactusB}, for an even cycle, we compute the $\gamma$-polynomial of the $h^*$-polynomial and the normalized volume of the symmetric edge polytope of type B.
\begin{example}
Let $C_{n}$ be an even cycle of length $n$.
It then from Example \ref{wheelexample} that
\[
g(C_n,4x)=L_n(4x)=\frac{
(1+\sqrt{1+16x})^n + (1-\sqrt{1+16x})^n
}{2^n}.
\]
From Theorem \ref{thm:cactusB} one has
\[
\gamma(\Bc_{C_n},x)=
 \frac{
(1+\sqrt{1+16x})^n + (1-\sqrt{1+16x})^n
}{2^n}  - (4x)^{\frac{n}{2} }.
\]
In particular, we obtain
\[
\Vol(\Bc_{C_n})=2^n \gamma(\Bc_{C_n},1/4)=(1+\sqrt{5})^n+(1-\sqrt{5})^n-2^n.
\]
\end{example}

Theorem \ref{thm:cactusB} generalizes the following result.
\begin{corollary}[{\cite[Proposition 3.5]{OTinterior}}]
\label{forest}
	Let $G$ be a forest.
	Then one has
	\[
	\gamma(\Bc_G,x)=g(G,4x).
	\]
	In particular, $h^*(\Bc_G,x)$ is real-rooted.
\end{corollary}

Finally, we show that for a forest $G$, $\gamma (\Bc_G,x)$ coincides with the $f$-polynomial of a flag simplicial complex. Namely, Nevo--Petersen's conjecture holds for any flag unimodular triangulation of the boundary $\partial \Bc_G$ in this case.

\begin{theorem}
\label{thm:NPB}
Let $G$ be a forest.
Then the $\gamma$-polynomial of $h^*({\mathcal B}_G,x)$
 coincides with the $f$-polynomial of a flag simplicial complex.
\end{theorem}

\begin{proof}
It follows from Corollary \ref{forest} that the $\gamma$-polynomial of $h^*({\mathcal B}_G,x)$
is $g(G,4x)$.
Thus Propositions~ \ref{prop:matchflag} and \ref{fmx} guarantee that $g(G,4x)$ is the $f$-polynomial of a flag simplicial complex.
\end{proof}

\begin{remark}\label{rem:B}
	It follows from the proof of \cite[Theorem 2.6]{OTinterior} that for any forest $G$, $\partial \Bc_G$ has a flag unimodular triangulation.
\end{remark}



%

\end{document}